\documentclass{amsart}
\usepackage{amscd,amsmath,amssymb,amsthm,amsfonts,epsfig,graphics}
\usepackage[all]{xy}

\newtheorem{theorem}{Theorem}
\newtheorem{lemma}{Lemma}

\newtheorem{corollary}{Corollary}

\theoremstyle{remark}
\newtheorem{remark}{Remark}

\newcommand{\te}{Teich\-m\"ul\-ler}
\newcommand{\h}{\mathbb{H}}
\newcommand{\z}{\mathbb{Z}}

\newcommand{\re}{\widehat{\mathbb{R}}}

\newcommand{\ud}{\mathbb{D}}
\newcommand{\uc}{\mathbb{S}^{1}}

\newcommand{\f}{\mathcal{F}}
\newcommand{\x}{\mathcal{X}}
\newcommand{\lam}{\mathcal L}
\numberwithin{equation}{section}

\begin{document}

\title[characterization of asymptotic Teichm\"uller space through shears]{Characterization of the asymptotic Teichm\"uller space of
the open unit disk through shears}

\author{Jinhua  Fan$^*$}
\address{Department of Applied Mathematics\\
Nanjing University of Science and Technology\\
Nanjing 210094, PRC} \email[]{jinhuafan@hotmail.com}

\author{Jun Hu$^{**}$}
\address{Department of Mathematics\\
Brooklyn College of CUNY\\
Brooklyn, NY 11210\\
and\\
Ph.D. Program in Mathematics\\
Graduate Center of CUNY\\
365 Fifth Avenue, New York, NY 10016}
\email[]{junhu@brooklyn.cuny.edu or JHu1@gc.cuny.edu}

\thanks{$^*$This work is supported by NNSF of China (No.
11201228). $^{**}$The work is partially supported by PSC-CUNY grants
and Brooklyn College Provost's Office for reassigned time in Spring
2013.}




\begin{abstract} We give a parametrization to the
asymptotic \te\ space $AT(\ud)$ of the open unit disk $\ud $
through equivalent classes of shear functions induced by
quasisymmetric homeomorphisms on the Farey tesselation of $\ud $.
Then using the parametrization, we define a new metric on $AT(\ud
)$. Two other related metrics are also introduced on $AT(\ud)$ by
using cross-ratio distortions or quadrilateral dilatations under
the boundary maps on degenerating sequences of quadruples or
quadrilaterals. We show that the topologies induced by the three
metrics are equivalent to the one induced by the \te\ metric on
$AT(\ud )$. Before proving our main results, we revisit and
rectify a mistake in the proof in \cite{Saric2} on the
characterization of quasisymmetric homeomorphisms in terms of
shear functions.
\end{abstract}

\maketitle

\section{Introduction}
Let $\ud$ be the open unit disk in the complex plane and centered
at the origin, and let $\uc$ be the boundary of $\ud $. Denote by
$T(\ud )$ the universal \te\ space and $AT(\ud )$ the asymptotic
\te\ space of $\ud $. Recently, two characterizations of the \te\
topology on $T(\ud )$ are given in \cite{MiyachiSaric} and
\cite{Saric2}: one uses a uniform weak$^*$ topology defined on the
Thurston parametrization of $T(\ud )$ (\cite{Thurston} or see
elaborations and further investigations in \cite{GardinerHuLakic},
\cite{Hu}, and \cite{Saric1}) - the space of all Thurston-bounded
measured geodesic laminations $\lam$ on $\ud $; the other is
comprised of a parametrization of $T(\ud )$ by shear functions $s$
induced by quasisymmetric homeomorphisms of $\uc $ on the Farey
tesselation of $\ud$ and a metric defined on the parametrization.
In the following, by a topological characterization of $AT(\ud )$
we mean a parametrization of $AT(\ud )$ and a topology or metric
on the parametrization equivalent to or with the induced topology
equivalent to the \te\ topology on $AT(\ud )$. Very recently, a
topological characterization of $AT(\ud )$ is given in
\cite{FanHu} by using an asymptotically uniform weak$^*$ topology
on a parametrization of $AT(\ud )$ by equivalent classes of
laminations $\lam$ in the Thurston parametrization of $T(\ud )$.
In this paper, we give a topological characterization of $AT(\ud
)$ by introducing a metric on a parametrization of $AT(\ud )$ by
equivalent classes of the shear functions $s$ representing $T(\ud
)$.

There are different models for $T(\ud )$ and $AT(\ud )$. In this
paper, the main model we use for $T(\ud )$ is the collection
$QS(\uc )$ of all quasisymmetric homeomorphisms of $\uc$ fixing
three points $1,-1$ and $i$; correspondingly, the one we use for
$AT(\ud)$ is the quotient space $S(\uc )\backslash QS(\uc )$,
where $S(\uc )$ is the collection of all symmetric homeomorphisms
in $QS(\uc )$. Before we give the statements of our theorems, we
first recall some background.

Consider $\ud $ as the hyperbolic plane. The Farey tesselation
$\mathcal{F}$ of $\ud $ is a locally finite idea triangulation of
$\ud $ that is invariant under the group of the isometries of $\ud
$ generated by the reflections in the hyperbolic metric with
respect to the geodesics of $\mathcal{F}$ (see Section 4 for the
details). The endpoints of the geodesics in $\mathcal{F}$ are
called the tips of $\mathcal{F}$ and we denote the collection of
all tips by $\mathcal{P}$. Each homeomorphism $h$ of $\uc $
induces a real-valued function $s_h:\mathcal{F}\rightarrow
\mathbb{R}$ as follows, which is called the {\em shear function}
or {\em coordinate} of $h$. Given each edge $e\in \f$, let
$(\Delta , \Delta _1)$ be the pair of the two adjacent triangles
of $\f$ sharing a common boundary at $e$. Then $s_h(e)$ is defined
to be the shear of the image pair $(h(\Delta ), h(\Delta_1))$,
that is the signed hyperbolic distance between the orthogonal
projections of the third vertices of $h(\Delta )$ and $h(\Delta
_1)$ to their common edge (see Section 4 for its interpretation
through the cross-ratio distortion under $h$ on a quadruple of
four points on $\uc$).

Let $\{e_n^p\}_{n\in \z}$ be a fan in $\f$ with tip $p$ (see
Section 4 for the definition), and let $s_n^p=s_h(e_n^p)$. Define
\begin{equation}\label{defofs-pmk}
s(p;m,k)=
e^{s_m^p}\frac{1+e^{s_{m+1}^p}+e^{(s_{m+1}^p+s_{m+2}^p)}+\cdots
+e^{(s_{m+1}^p+s_{m+2}^p+\cdots+s_{m+k}^p)}}
{1+e^{-s_{m-1}^p}+e^{-(s_{m-1}^p+s_{m-2}^p)}+\cdots
+e^{-(s_{m-1}^p+s_{m-2}^p+\cdots+s_{m-k}^p)}}
\end{equation}
for any $m, k\in \z$. This quantity can be interpreted as the
cross-ratio distortion of $h$ on a quadruple of four points on
$\uc $ (see Section 4). Then define
$$M_s(p)=\sup _{m, k\in \z}| \log s(p;m, k)| \;\;{\rm and }\;\; M_s=\sup _{p\in \mathcal{P}}M_s(p).$$
We call $M_s$ the {\em shear norm} of $h$. We say that a circle
homeomorphism $h$ of $\uc $ has a {\em finite shear norm} if
 $M_s$ is finite.

It is stated in Theorem A of \cite{Saric2} that a real-valued
function $s:\f\rightarrow \mathbb{R}$ is equal to the shear
function $s_h$ of a quasisymmetric homeomorphism $h$ of $\uc $ if
and only if $s$ has a finite shear norm. Let $\x$ be the
collection of all real-valued functions $s:\f \rightarrow
\mathbb{R}$ with finite shear norms. Then $\x$ gives a
parametrization of $T(\ud )$. Given two elements $s_1, s_2 \in
\x$, define
$$M_{s_1, s_2}(p)=\sup _{m, k\in \z}|\log\frac{s_1(p;m,k)}{s_2(p;m,k)}| \;\;{\rm and }\;\;d_S(s_1, s_2)=\sup _{p\in \mathcal{P}}M_{s_1, s_2}(p).$$
Then $d_S$ defines a metric on $\x$, called the {\em shear
metric}. In Theorem B of \cite{Saric2}, \v{S}ari\'c showed that
the parametrization $\x$ of $T(\ud )$ equipped with the metric
$d_S$ is topologically equivalent to $T(\ud )$ under the \te\
metric $d_T$. More explicitly, he proved that a sequence of points
converges to a point in $T(\ud )$ under $d_T$ if and only if the
sequence of the corresponding shear functions converges to the
shear function of the limiting point under $d_S$.

The first goal of this paper is to give a parametrization to the
asymptotic \te\ space $AT(\ud)$ via equivalent classes of shear
functions, define a new metric on the parametrization, and show
that the parametrization with the metric is a topological
characterization of $AT(\ud )$ under the \te\ metric $d_{AT}$.

We first introduce a pseudo metric $d_{AS}$ on $\mathcal{X}$
(defined at the beginning of Section 5). This pseudo metric
introduces an equivalent relation on $\mathcal{X}$ when two
elements $s$ and $s^\prime $ of $\mathcal{X}$ are defined to be
equivalent if $d_{AS}(s, s^\prime )=0$. We denote by $[s]$ the
equivalent class of $s$ and by $\mathcal{AX}$ the quotient space
under this equivalent relation. Then $d_{AS}$ induces a metric on
$\mathcal{AX}$, which is continued to be denoted by $d_{AS}$ and
called the {\em asymptotic shear metric} on $\mathcal{AX}$. In the
following two theorems, we show that $\mathcal{AX}$ is a
parametrization of $AT(\ud )$ and $(\mathcal{AX}, d_{AS})$ is
topologically equivalent to $(AT(\ud ), d_{AT})$, where $d_{AT}$
is the \te \ metric on $AT(\ud )$.

\begin{theorem} \label{charactrization-via-shears}
Let $h,h^\prime\in T(\ud)$ and let $s$ and $s^\prime$ be the
shear functions induced by $h$ and $h^\prime$ respectively. Then
$[h]=[h^\prime]$ if and only if $d_{AS}(s,s^\prime)=0$, where
$[h]$ denote a point in the asymptotic \te\ space $AT(\ud).$
\end{theorem}

\begin{theorem}\label{top-equiv}
Let $[h]\in AT(\ud)$ and let $\{[h_n]\}_{n=1}^\infty$ be a
sequence of points in $AT(\ud)$. Assume that $s_n$ and $s$ are the
shear functions induced by $h_n$ and $h$ respectively. Then
 $d_{AT}([h_n],
[h])\rightarrow 0$ if and only if $d_{AS}([s_n],[s])\rightarrow
0$.
\end{theorem}

\medskip
The Douady-Earle extensions \cite{DouadyEarle} of circle
homeomorphisms of $\uc $ have played an important role in the
study of \te\ spaces and asymptotic \te\ spaces. For some recent
applications of these extensions, we refer to
\cite{EarleMarkovicSaric}, \cite{Markovic}, \cite{MiyachiSaric},
\cite{Saric2} and \cite{FanHu}. Further investigation on
regularities or generalizations of the Douady-Earle extensions to
large classes of circle maps have been developed in
\cite{HuMuzician1}-\cite{HuMuzician5} and \cite{HuPal}.

The conformal naturality and one local regularity near the origin
of the Douady-Earle extensions together play a crucial role in the
proof of Theorem A in \cite{Saric2}. That local regularity is
summarized as Lemma 2.2 in \cite{Saric2}, which states as follows.
{\em Let $\re=\mathbb{R}\cup \{\infty \}$ be the extended real
line, which represents the boundary of the hyperbolic plane when
it takes the upper half plane $\h$ as a model. Let
$\{h_n\}_{n=1}^{\infty }$ be a sequence of orientation-preserving
homeomorphisms $h_n$ of $\re $ fixing three points $0$, $1$ and
$\infty $, and let $\mu _n$ be the Beltrami coefficient of the
Douady-Earle extension $ex(h_n)$ of $h_n$ for each $n$. If there
exists $c_0\ge 1$ such that
$$-c_0\le h_n(-1)\le -\frac{1}{c_0}$$ for all $n$, then there exists a
neighborhood $U$ of the imaginary unit $i\in \h$ and a constant
$0<c<1$ such that $$||\mu_n|_{U}||_{\infty }\le c<1$$ for all $n$.}

Theorem A of \cite{Saric2} gives a necessary and sufficient
condition for $s:\mathcal{F}\rightarrow \mathbb{R}$ to be induced
by a quasisymmetric (resp. symmetric) homeomorphism of $\uc $. The
necessity of the condition can be easily obtained, see Remark 1 in
Section 4 (resp. Remark 2 in Section 5). The main part of the
proof of Theorem A is to show that $M_s<\infty $ is sufficient for
$s$ to be induced by a quasisymmetric homeomorphism of $\uc $.
This part is proved there by applying the above local regularity
of Douady-Earle extensions, that is Lemma 2.2 in \cite{Saric2}.
Counter-examples have been found for that lemma. Using an
algorithm called the MAY operator \cite{Abikoff}, the first
counter-example was tested numerically and then proved in
\cite{HuMuzician3}. In Section 3 of this paper, we provide a
slightly different example without involving any explicit formulas
so that it can be verified relatively easily. This means that
there is a mistake or gap in \cite{Saric2} in the proof of the
sufficient condition for $s$ to be induced by a quasisymmetric
homeomorphism. The second gaol of this paper is to rectify that
mistake or bridge the gap in the proof caused by the
counter-examples. We thereafter develop in Section 3 a new lemma
for being used to bridge the gap. Then in Section 4, we apply our
new lemma to show how the gap is bridged by proving an improved
version of the sufficient condition as follows.

\begin{theorem}\label{improvedversion}
 Let $h$ be a homeomorphism of $\re$ and let
 $s$ be the shear function induced by $h$. Then
$$K(ex(h))<C(M_s);$$
that is, the maximal dilatation of the Douady-Earle extension of
$h$ is bounded from above by a positive constant only depending on
the shear norm of $h$.
\end{theorem}

In the course of developing and proving our Theorems
\ref{charactrization-via-shears} and \ref{top-equiv}, two related
metrics on $AT(\ud)$ come into place naturally: one is defined by
using cross-ratio distortions under the boundary maps on
degenerating sequences of quadruples and the other is defined by
using quadrilateral dilatations under the boundary maps on
degenerating sequences of quadrilaterals. Tracing back to similar
metrics introduced on the universal \te\ space $T(\ud )$, in the
remaining part of this introduction section we first recall the
two metrics on $T(\ud)$ defined by using cross-ratio distortions
and quadrilateral dilatations respectively; then we define the
corresponding metrics on $AT(\ud )$.

For each $h\in T(\ud)$, let
$$K_e(h)=\inf_{f|_{\uc}=h}K(f),$$
where the infimum is taken over all quasiconformal extensions of
$h$ to $\ud $. Then the \te \ metric $d_T$ is defined as
$$d_T(h_1, h_2)=\frac{1}{2} \log K_e(h_2\circ h_1^{-1})$$
for any two points $h_1, h_2\in T(\ud)$.

Given a quadruple $Q=\{a, b, c, c\}$ of four points $a, b, c$ and
$d$ on $\uc$ arranged in counter-clockwise order, let $M(Q)$
denote the conformal modulus of the (topological) quadrilateral
$\ud(a, b, c, d)$. Each orientation-preserving homeomorphism $h$
of $\uc $ maps $\ud(a, b, c, d)$ to another quadrilateral
$\ud(h(a), h(b), h(c), h(d))$. Correspondingly, we let $M(h(Q))$
denote the conformal modulus of $\ud(h(a), h(b), h(c), h(d))$.
Given $h\in T(\ud)$, the maximal quadrilateral dilatation of $h$
is defined as
$$K_m(h)=\sup_Q\{\frac{M(h(Q))}{M(Q)} \},$$
where the supremum is taken over all quadruples $Q$. Given any two
points $h_1, h_2\in T(\ud)$, we define
$$d_M(h_1, h_2)=\frac{1}{2} \log K_m(h_2\circ h_1^{-1}).$$
By definition, it is clear that
$$d_M(h_1, h_2)=\frac{1}{2}\sup_Q|\log\frac{M(h_2(Q))}{M(h_1(Q))}|.$$
It is well known that
$$\frac{1}{K(f)}\leq\frac{M(h(Q))}{M(Q)}\leq
K(f)$$  for any quasiconformal extension $f$ of $h$ and any
quadruple $Q$. It follows that for any $h\in T(\ud)$, $$K_m(h)\leq
K_e(h),$$ and for any two points in $[h_1], [h_2]\in T(\ud)$,
$$d_M(h_1, h_2)\leq d_T(h_1, h_2).$$

Given a homeomorphism $h$ of $\uc $ and a quadruple $Q=\{a, b, c,
d\}$ of four points on $\uc $ arranged in the counter-clockwise
order, let $h(Q)$ be the image quadruple $\{h(a),h(b),h(c),h(d)\}$
and let $cr(Q)$ denote the following cross ratio of $Q$:
\begin{equation}\label{defofcr}
cr(Q)=\frac{(b-a)(d-c)}{(c-b)(d-a)}.
\end{equation}
Given any two
points $h_1, h_2\in T(\ud)$, we define
$$d_C(h_1, h_2)= \sup_{cr(Q)=1}|\log\frac{cr(h_2(Q))}{cr(h_1(Q))}|.$$

One has known that $d_M$ and $d_C$ are metrics on $T(\ud )$ and
the three metrics $d_T$ and $d_M$ and $d_C$ on $T(\ud )$ are
topologically equivalent. Since we cannot find references
explicitly giving these two statements and their proofs are not
long, we state them as the following folklore theorems and provide
their proofs in Section 2.
\begin{theorem}[Folklore Theorem 1] \label{d-Mandd-C}
Both $d_M$ and $d_C$ define metrics on $T(\ud)$.
\end{theorem}

\begin{theorem}[Folklore Theorem 2]
\label{equi-of-three-metrics-on-T} The three metrics $d_T$, $d_M$
and $d_C$ on $T(\ud )$ are topologically equivalent.
\end{theorem}

An orientation-preserving homeomorphism $h$ of $\uc$ is said to be
{\em symmetric} if
$$(1+\delta(x,t))^{-1} \leq \frac{|h(x+t)-h(x)|}{|h(x)-h(x-t)|}\leq 1+\delta(x, t),
$$
where $\delta(x, t)\rightarrow 0$ uniformly in $x\in \mathbb{S}^1$
as $t\rightarrow 0$. By letting $T_0(\ud)$ be the subspace of
$T(\ud)$ whose elements are symmetric homeomorphisms, the
asymptotic \te\ space can be expressed as the quotient space
$$AT(\ud)= T(\ud)/ T_0(\ud).$$ Given any $h\in T(\ud)$, we denote by $[h]$ the corresponding point in $AT(\ud)$.
This space was studied by Gardiner and Sullivan in
\cite{GardinerSullivan} and the asymptotic \te \ spaces of Riemann
surfaces were studied in \cite{EarleGardinerLakic2000} \
\cite{EarleGardinerLakic2004},\ \cite{EarleMarkovicSaric},\
\cite{GardinerLakic} and etc.

Parallel to $d_T, d_M$ and $d_C$  on $T(\ud)$, three metrics can
be defined on $AT(\ud)$.

Let $f$ be a quasiconformal conformal mapping on $\ud$. The {\em
boundary dilatation}  $H(f)$ of $f$ is defined as
$$H(f)=\inf_E K(f|_{\ud\backslash E}),$$ where the infimum  is taken over all compact subsets $E$ of $\ud$.
Given any point $[h] \in AT(\ud)$, the boundary dilatation of
$[h]$ is defined as
$$K_{ae}([h])=\inf_{f|_{\uc}\in [h]} H(f).$$
 The  \te \ distance between two points $[h_1],[h_2]\in
AT(\ud)$ is defined as
$$d_{AT}([h_1],[h_2])=\frac{1}{2}\log K_{ae}([h_2\circ (h_1)^{-1}]).$$

Given a topological quadrilateral $Q=\ud(a,b,c,d)$ or  quadruple
$Q=\{a,b,c,d\}$,  we define the {\em minimal scale} of $Q$ as
$$s(Q)=\min\{|a-b|,|b-c|,|c-d|,|d-a|\}.$$
Then a sequence $\{Q_n=\ud(a_n,b_n,c_n,d_n)\}_{n=1}^\infty$ of
topological quadrilaterals
 is said to be {\em degenerating}
if $s(Q_n)\rightarrow 0$ as $n\rightarrow \infty.$ A sequence
$\{Q_n=\{a_n ,b_n,c_n,d_n\}\}_{n=1}^\infty$  of
 quadruples is said to be {\em degenerating}
if $cr(Q_n)=1$ for each $n$ and $s(Q_n)\rightarrow 0$ as
$n\rightarrow \infty $. For any $[h]\in AT(\ud)$, we define the
asymptotically maximal quadrilateral dilatation of $[h]$ as
$$K_{am}([h])=\sup_{\{Q_n\}}\limsup_{n\rightarrow \infty}\{\frac{M(h(Q_n))}{M(Q_n)} \},$$
where the supremum is taken over all degenerating sequences of
topological quadrilaterals $\{Q_n\}_{n=1}^\infty$. Given any two
points $[h_1],[h_2]\in AT(\ud)$, we define
 $$d_{AM}([h_1], [h_2])=\frac{1}{2}
 \log K_{am}([h_2\circ h_1^{-1}]).$$
Clearly,
$$d_{AM}([h_1], [h_2])=\frac{1}{2}\sup_{\{Q_n\}}\limsup_{n\rightarrow \infty}|\log\frac{M(h_2(Q_n))}{M(h_1(Q_n))}|,$$
where the supremum is taken over all degenerating sequences
$\{Q_n\}_{n=1}^\infty$ of topological quadrilaterals.

From the proof of Theorem 1 of  \cite{wushengjian}, it follows
that for any $[h]\in AT(\ud )$,
$$K_{ae}([h])\geq K_{am}([h]).$$
Thus, for any two points $[h_1],[h_2]\in AT(\ud)$,
$$d_{AM}([h_1], [h_2])\le d_{AT}([h_1], [h_2]).$$

Given any two points $[h_1], [h_2]\in AT(\ud)$, we define
$$d_{AC}([h_1], [h_2])= \sup_{\{Q_n\}}\limsup_{n\rightarrow \infty}|\log \frac{cr(h_2(Q_n))}{cr(h_1(Q_n))}|,$$
where the supremum is taken over all degenerating sequences
$\{Q_n\}_{n=1}^\infty$ of quadruples.

In Section 2, we also show the following two theorems.

\begin{theorem}\label{d-AMandd-AC}
Both $d_{AM}$ and $d_{AC}$ define metrics on $AT(\ud)$.
\end{theorem}

\begin{theorem}\label{equi-of-three-metrics-on-AT}
The three metrics $d_{AT}$, $d_{AM}$ and $d_{AC}$ on $AT(\ud )$
are topologically equivalent.
\end{theorem}

The paper is arranged as follows. We first give proofs of Theorems
4-7 in the second section. In the third section, we provide a
counter-example to Lemma 2.2 in \cite{Saric2} and introduce a new
lemma. Then in the fourth section, we apply the new lemma to
bridge the gap in the proof of Theorem A  in \cite{Saric2} by
showing Theorem \ref{improvedversion}. The proofs of Theorems
\ref{charactrization-via-shears} and \ref{top-equiv} are given in
the fifth section.

\medskip
\noindent {\bf Acknowledgement:} The authors wish to thank
Professors Frederick Gardiner for his comments and suggestions on
the writing of the introduction.

\section{Topological equivalence of three metrics on $T(\ud )$ and $AT(\ud )$}

In this section, we show Theorems 4-7. Let $\rho_{(0,-1)}$ be the
Poincar$\mathrm{\acute{e}}$ density on
$\mathbb{C}\backslash\{0,-1\}$ and let
$d_{\rho_{(0,-1)}}(\alpha,\beta)$ be the distance under
$\rho_{(0,-1)}$ between two points $\alpha,\beta\in
\mathbb{C}\backslash\{0,-1\}$. The following lemma was proved in
\cite{wuchong}.

\begin{lemma} \label{wuchong} For any quadruples $Q_1$ and $Q_2$ on $\uc$,
\begin{equation}
d_{\rho_{(0,-1)}}(cr(Q_1), cr(Q_2))=|\log\frac{M(Q_1)}{M(Q_2)}|.
\end{equation}
\end{lemma}

\begin{proof} [Proof of Theorem 4] It is easy to check that $d_Q$
and $d_C$ are symmetric and satisfy the triangle inequality. To
prove  $d_M$ and $d_C$ to be metrics, we only need to check that
$d_M(h_1,h_2)=0$ or $d_C(h_1,h_2)=0$ implies $h_1=h_2.$

We first show that $d_M(h_1,h_2)=0$ implies $d_C(h_1,h_2)=0$. If
on the contrary, $d_C(h_1,h_2)>0$, then there exist $m>0$ and a
quadruple $Q=(a,b,c,d)\in\uc$ with $cr(Q)=1$ such that
$$\frac{cr(h_1(Q))}{cr(h_2(Q))}>1+m.$$ Then Lemma 1 implies
$$|\log\frac{M(h_1(Q))}{M(h_2(Q))}|>0,$$ which is a contradiction
to $d_M(h_1,h_2)=0$.

Now we show that  $d_C(h_1,h_2)=0$ implies $h_1=h_2.$ Let
$A(z)=-\frac{z-i}{z+i}$, which is a M$\mathrm{\ddot{o}}$bius
transformation from $\h$ to $\ud$, and let $\tilde{h}_1=
A^{-1}\circ h_1\circ A$ and $\tilde{h}_2= A^{-1}\circ h_2\circ A$.
Then $d_C(h_1,h_2)=0$ means $
cr(\tilde{h}_1(Q))=cr(\tilde{h}_2(Q))$ for any quadruple
$Q=\{a,b,c,d\}\subset \re$ with $cr(Q)=1$. Since $\tilde{h}_1$ and
$\tilde{h}_2$ fix $0,1$ and $\infty$, by applying the condition $
cr(\tilde{h}_1(Q))=cr(\tilde{h}_2(Q))$ repeatedly on appropriate
quadruples $Q\subset \mathbb{Z}$, we first obtain that
$\tilde{h}_1(x)=\tilde{h}_2(x)$ for all $x\in \mathbb{Z}$. Then by
using quadruples $Q\subset \frac{\mathbb{Z}}{2}$, we conclude that
$\tilde{h}_1(x)=\tilde{h}_2(x)$ for all $x\in
\frac{\mathbb{Z}}{2}$. Inductively,
$\tilde{h}_1(x)=\tilde{h}_2(x)$ for all $x\in
\frac{\mathbb{Z}}{2^n}$ for each $n\in \mathbb{N}$. Since $\cup
_{n=1}^{\infty }\frac{\mathbb{Z}}{2^n}$ is dense in $\re$, it
follows that $\tilde{h}_1(x)=\tilde{h}_2(x)$ for all $x\in \re$;
that is, $h_1=h_2.$
\end{proof}

\begin{proof}[Proof of Theorem 5] Let $h\in T(\ud )$ and let
$\{h_n\}_{n=1}^{\infty }$ be a sequence of points in $T(\ud )$.

At first, since $d_M(h_n,h)\leq d_T(h_n,h)$, it follows that
$d_T(h_n,h)\rightarrow 0$ implies $d_M(h_n,h)\rightarrow 0$ as
$n\rightarrow \infty $.

Secondly, we show that $d_M(h_n,h)\rightarrow 0$ implies
$d_C(h_n,h)\rightarrow 0$ as $n\rightarrow \infty$. If on the
contrary this is not true, then there exists a subsequence of
$\{h_n\}_{n=1}^\infty$, denoted again by $\{h_n\}_{n=1}^\infty$
for simplicity, such that $d_C(h_n,h)>\varepsilon_0$ for all $n$
and some $\varepsilon_0>0$. Thus there is a sequence
$\{Q_n=\{a_n,b_n,c_n,d_n\}\}_{n=1}^\infty\subset\uc$ of quadruples
with $cr(Q_n)=1$, such that for all $n$,
$$\frac{cr(h_n(Q_n))}{cr(h(Q_n))}>e^{\varepsilon_0}.$$  Since
$h$ is a given quasisymmetric homeomorphism and $cr(Q_n)=1$,
$\{cr(h(Q_n))\}_{n=1}^\infty$ stays in a compact subset of the
positive half real axis.  Thus there exists $\varepsilon_0^\prime>0$
such that for all $n$,
$$d_{\rho_{(0,-1)}}(cr(h_n(Q_n),
cr(h(Q_n)))>\varepsilon_0^\prime.$$ Using Lemma 1, we conclude
that for all $n$,
$$|\log\frac{M(h_n(Q_n))}{M(h(Q_n))}| >\varepsilon_0^\prime.$$
This is a contradiction to the assumption that
$d_Q(h_n,h)\rightarrow 0$ as $n\rightarrow \infty $.

Finally, we show that $d_C(h_n,h)\rightarrow 0$ implies
$d_T(h_n,h)\rightarrow 0$ as $n\rightarrow \infty $. If on the
contrary this is not true, then by passing to a subsequence we may
assume that $d_T(h_n,h)>\varepsilon _1$ for all $n$ and some
$\varepsilon _1>0$.

 Let $ex(h_n)$ be the Douady-Earle extension of $h_n$. Since $d_C(h_n,h)\rightarrow
0$ as $n\rightarrow \infty $, there exists $C>0$ such that for all
$n$ and any quadruple $Q$ with $cr(Q)=1$,
$$|\log cr(h_n(Q))|<C.$$
Using the main theorem in \cite{HuMuzician1}, there exists $K_0>1$
such that for all $n$, $$K(ex(h_n))\leq K_0.$$ Thus $\{h_n\}$ is a
bounded sequence in $T(\ud )$. Passing to subsequences, we may
assume that $\{h_n\}_{n=1}^{\infty }$ converges to a
quasisymmetric homeomorphism $h^\prime$ uniformly on $\uc$ and
also converges to a point $h''$ in $T(\ud )$ under the metric
$d_T$. Note that $h'$ and $h''$ fix the same three points as each
$h_n$ does. Using the conclusions in the previous two steps, we
obtain $d_C(h_n, h'')\rightarrow 0$ as $n\rightarrow \infty $.
Then for each quadruple $Q$ with $cr(Q)=1$,
$$cr(h''(Q))=\lim _{n\rightarrow \infty }cr(h_n(Q))=cr(h'(Q)).$$
Therefore, $h'=h''$. Since $d_c(h_n, h)\rightarrow 0$ as
$n\rightarrow \infty $, we concluded that $h=h'=h''$. Thus,
$d_T(h_n, h)\rightarrow 0$ as $n\rightarrow \infty $. This is a
contradiction to the assumption that $d_T(h_n,h)>\varepsilon _1$
for all $n$.

We complete the proof.
\end{proof}

\begin{proof} [Proof of Theorem  6]
Using definitions, it is easy to check that $d_{AM}$ and $d_{AC}$
are symmetric and satisfy the triangle inequality. To prove
$d_{AM}$ and $d_{AC}$ to be metrics, it remains to show that
$d_{AM}([h_1],[h_2])=0$ or $d_{AC}([h_1],[h_2])=0$ implies
$[h_1]=[h_2].$

Using the same argument to show that  $d_M(h_1,h_2)=0$ implies
$d_C(h_1,h_2)=0$ in the proof of Theorem 4, one can see that
$d_{AM}([h_1],[h_2])=0$ implies $d_{AC}([h_1],[h_2])=0$.

In the following, we apply some properties of Douady-Earle
extensions to show that $d_{AC}([h_1],[h_2])=0$ implies
$[h_1]=[h_2].$ Suppose on the contrary that $[h_1]\neq[h_2],$
which means $h_1\circ (h_2)^{-1} $ is not symmetric. Let $ex(h_1)$
and $ex(h_2)$ be the Douady-Earle extensions of $h_1$ and $h_2$
respectively. Using a result of \cite{EarleMarkovicSaric}, we know
that $ex(h_1)\circ (ex(h_2))^{-1}$ is not asymptotic conformal on
$\ud$, which means that there exist a constant $\varepsilon_0>0$
and a sequence $\{D_n\}_{n=1}^\infty$ of hyperbolic disks  in
$\ud$ of diameter $1$ with the Euclidean distance from $D_n$ to
$\uc$ approaching $0$ as $n\rightarrow \infty $ such that
\begin{equation}\label{2-2}
\|\mu_{ex(h_1)}|_{D_n}-\mu_{ex(h_2)}|_{D_n}\|_{L^\infty}\ge
\varepsilon_0
\end{equation}
for all $n$. Let $D_0$ be the hyperbolic disk on $\ud$ of diameter 1
and centered in $0$, and assume that $\gamma_n\in M\ddot{o}b(\ud)$
and $\gamma_n(D_0)=D_n$, where  $ M\ddot{o}b(\mathbb{D})$ is the
group of all M$\mathrm{\ddot{o}}$bius transformation preserving $\ud
$. Let $A_{1n}$ and $A_{2n}$
 $\in M\ddot{o}b(\ud)$ such that $A_{1n}\circ h_1\circ \gamma_n$ and
 $A_{2n}\circ h_2\circ \gamma_n$ fix $1,-1,i$ for all $n$. Given
 any quadruple $Q\in \uc$ with $cr(Q)=1$, from the fact that $D_n$
 converges to the boundary $\uc $, one can see that
 $s(\gamma_n(Q))\rightarrow 0$ as $n\rightarrow \infty $. Using the
 assumption that
 $d_{AC}([h_1],[h_2])=0$, we obtain
 \begin{equation}\label{2-3}
 \lim_{n\rightarrow \infty}\frac{cr(A_{1n}\circ h_1\circ \gamma_n(Q))}{cr(A_{2n}
 \circ h_2\circ \gamma_n(Q))}=1.
 \end{equation}
 Let $ex(A_{1n}\circ h_1\circ \gamma_n)$ and  $ex(A_{2n}\circ
h_2\circ
 \gamma_n)$ be the Douady-Earle extensions
of $A_{1n}\circ h_1\circ \gamma_n$ and $A_{2n}\circ h_2\circ
 \gamma_n$ respectively. Since these quasiconformal mappings fix
 three common points and have constant maximal dilatations,
 passing to subsequences, we may assume that they converge to
 two quasisymmetric homeomorphisms $\widetilde{h}_1$ and $\widetilde{h}_2$ uniformly on  $\uc$.
Using (\ref{2-3}) and the convergences, we obtain
$$cr(\widetilde{h}_1(Q))=cr(\widetilde{h}_2(Q)) $$
for any quadruple $Q\in \uc$ with $cr(Q)=1$. Using the fact
$\widetilde{h}_1$ equals to $\widetilde{h}_2$  at three fixed
point and Theorem 4,  we conclude that
$\widetilde{h}_1=\widetilde{h}_2$. Using a convergence property of
Douady-Earle extensions, $Belt(ex(A_{1n}\circ h_1\circ \gamma_n))$
and $Belt(ex(A_{2n}\circ h_2\circ \gamma_n))$ converge to
$Belt(ex(\widetilde{h}_1))$ and $Belt(ex(\widetilde{h}_2))$
uniformly on $D_0$ respectively; that is,
\begin{equation}\label{2-4}
||\mu_{ex(A_{1n}\circ h_1\circ
\gamma_n)}|_{D_0}-\mu_{ex(A_{2n}\circ h_2\circ
\gamma_n)}|_{D_0}||_{L^\infty }\rightarrow 0.
\end{equation}
On the other hand, by (\ref{2-2}) and  the conformal naturality of
Douady-Earle extensions,
\begin{eqnarray*}
&  & ||\mu_{ex(A_{1n}\circ h_1\circ
\gamma_n)}|_{D_0}-\mu_{ex(A_{2n}\circ h_2\circ
\gamma_n)}|_{D_0}||_{L^\infty } \\
& = &||\mu_{ex(h_1)}|_{D_n}-\mu_{ex(h_2)}|_{D_n}||_{L^\infty }\geq
\varepsilon_0.
\end{eqnarray*}
This is a contradiction to (\ref{2-4}). Thus, $h_1\circ (h_2)^{-1}
$ is symmetric and hence $[h_1]=[h_2]$.
\end{proof}

\begin{lemma}
 \cite{HuMuzician2} Let $h$ be an orientation-preserving
homeomorphism of $\uc$ and $ex(h)$ be the Douady-Earle extension
of $h$ to the closed unit disk $\ud$. Let $p\in \uc$ and $I_p$ be
an open arc on $\uc$ containing $p$ and symmetric with respect to
$p$. Assume that
$$\|h|_{I_p}\|_{cr}=\sup_{cr(Q)=1}|\log cr (h(Q))|<\infty,$$ where
the supremum is over all quadruple $Q\in I.$ Then there exists an
open hyperbolic half plane $U_p$ with $p$ at the middle of its
boundary on $\uc$ such that $$\log K(ex(h)|U_p)\leq
C_1\|h|_{I_p}\|_{cr} + C_2$$ for two universal positive constants
$C_1$ and $C_2$, where $K(ex(h)|U_p)$ is the maximal dilatation of
$ex(h)$ on $U_p$.
\end{lemma}

\begin{proof} [Proof of Theorem 7] Let $[h]\in AT(\ud )$ and
$\{[h_n]\}_{n=1}^{\infty }$ be a sequence of points in $AT(\ud)$.
We prove the equivalence of the topologies induced by the three
metrics $d_{AT}$, $d_{AM}$ and $d_{AC}$ in the same order as we
show the equivalence of the topologies induced by $d_{T}$, $d_M$
and $d_C$ in Theorem 5.

At first, as pointed out in the introduction,
$d_{AM}([h_n],[h])\leq d_{AT}([h_n],[h])$. Thus,
$d_{AT}([h_n],[h])\rightarrow 0$ implies
$d_{AM}([h_n],[h])\rightarrow 0$ as $n\rightarrow \infty $.

Secondly, we show $d_{AM}([h_n],[h])\rightarrow 0$ implies
$d_{AC}([h_n],[h])\rightarrow 0$ as $n\rightarrow \infty $. If on
the contrary this is not true, then there exists a subsequence of
$\{[h_n]\}_{n=1}^\infty$, denoted again by
$\{[h_n]\}_{n=1}^\infty$ for simplicity, such that
$d_{AC}([h_n],[h])>\varepsilon_0$ for all $n$ and some
$\varepsilon_0>0$. Then, for each $n$, there is a degenerating
sequence
$\{Q_{n_j}=\{a_{n_j},b_{n_j},c_{n_j},d_{n_j}\}\}_{j=1}^\infty$ of
quadruples  and $\varepsilon_0>0$ such that, for each $n$,
$$\lim_{j\rightarrow \infty} \frac{cr(h_n(Q_{n_j}))}{cr(h(Q_{n_j}))}>e^{\varepsilon_0}.$$
Thus, we can choose a degenerating sequence
$\{Q_{n}=\{a_{n},b_{n},c_{n},d_{n}\}\}_{n=1}^\infty$ of quadruples
such that for each $n$,
$$ \frac{cr(h_n(Q_{n}))}{cr(h(Q_{n}))}>e^{\varepsilon_0}.$$
Using the same argument in the second step in the proof of Theorem
5, we obtain
$$d_{\rho_{(0,-1)}}(cr(h_n(Q_n),
cr(h(Q_n)))>\varepsilon_0^\prime$$ for all $n$ and some
$\varepsilon_0^\prime>0$. By Lemma 1, for all $n$,
$$|\log\frac{M(h_n(Q_n))}{M(h(Q_n))}| >\varepsilon_0^\prime,$$
which is a contradiction to the assumption that
$d_{AQ}([h_n],[h])\rightarrow 0$ as $n\rightarrow \infty $.

Finally, we show that $d_{AC}([h_n],[h])\rightarrow 0$ implies
$d_{AT}([h_n],[h])\rightarrow 0$. Suppose on the contrary that
$d_{AT}([h_n],[h])\nrightarrow 0$. Passing to a subsequence, we
may assume that, for all $n$,
\begin{equation}\label{2-5}
d_{AT}([h_n],[h])>\varepsilon>0.
\end{equation}
Let $ex(h_n)$ and $ex(h)$ be the Douady-Earle extensions of $h_n$
and $h$ respectively.  Using (\ref{2-5}), we can choose  a
sequence of hyperbolic disks $\{D_n\}_{n=1}^\infty\subset\ud$ of
diameter $1$ with the Euclidean distance from $D_n$ to $\uc$
approaching $0$ as $n\rightarrow \infty $ such that, for all $n$,
\begin{equation}\label{2-6}
\|\mu_{ex(h_n)}|_{D_n}-\mu_{ex(h)}|_{D_n}\|_{L^\infty}\ge
\epsilon.
\end{equation}
 Let $D_0$ be the hyperbolic disk on $\ud$ of diameter 1 and
centered in $0$, and assume that $\gamma_n\in M\ddot{o}b(\ud)$ and
$\gamma_n(D_0)=D_n$. Let $A_{1n}$ and $A_{2n}$
 $\in M\ddot{o}b(\ud)$ such that $A_{1n}\circ h_n\circ \gamma_n$ and
 $A_{2n}\circ h\circ \gamma_n$ fix $1,-1,i$ for all $n$.
Now given a quadruple $Q$ with $cr(Q)=1$, one can check that
$s(\gamma (Q))\rightarrow 0$ as $n\rightarrow \infty $. Using the
assumption that $d_{AC}([h_n],[h])\rightarrow 0$ as $n\rightarrow
\infty $, we obtain
 \begin{equation}\label{2-7}
 \lim_{n\rightarrow \infty}|\frac{cr(A_{1n}\circ h_n\circ \gamma_n(Q))}{cr(A_{2n}\circ h\circ
 \gamma_n(Q))}|=1.
 \end{equation}

Since $d_{AC}([h_n],[h])\rightarrow 0$ as $n\rightarrow \infty$,
$\{d_{AC}([h_n],[Id])\}_{n=1}^\infty$ is bounded. Using the
definition of $d_{AC}$ and Lemma 2, we know that
$\{H(ex(h_n))\}_{n=1}^\infty$ is bounded. For each $n$, we can
choose $0<r_n<1$ such that
$\frac{1+\|\mu_n\|_\infty}{1-\|\mu_n\|_\infty}\leq H(ex(h_n))+1$,
where $\mu_n$ is defined as
$$
\mu_n(z)=\left\{\begin{array}{ll}\displaystyle
                    \mu_{ex(h_n)}(z), & r_n\leq|z|<1,\\
                    0, & \ \ |z|<r_n.
              \end{array}\right.$$
Let $f_n$ be the normalized (i.e., fixing three points $1$, $-1$
and $i$) quasiconformal homeomorphism of $\ud$ with the Beltrami
coefficient $\mu_n$, and $\widetilde{h}_n=f_n|_{\uc}$. Then
$\widetilde{h}_n\in[h]$ for each $n$ and
$\{K(ex(\widetilde{h}_n))\}_{n=1}^\infty$ is bounded. For
simplicity, we denote $\{[\widetilde{h}_n]\}_{n=1}^\infty$ by
$\{[h_n]\}_{n=1}^\infty$. In other words, by replacing
representatives of $[h_n]'s$, we may assume that
$\{K(ex(h_n))\}_{n=1}^\infty$ is bounded.

Since all $A_{1n}\circ h_n\circ \gamma_n$ and $A_{2n}\circ h\circ
\gamma_n$ are normalized to fix three points,
 passing to subsequences we may assume that $A_{1n}\circ h_n\circ \gamma_n$ and
$A_{2n}\circ h\circ \gamma_n$ converge uniformly to quasisymmetric
maps $h^\ast$ and $\widehat{h}^\ast$ respectively on $\uc$. Using
(\ref{2-7}) and the convergences, we obtain
$$ cr(h^\ast(Q))=cr(\widehat{h}^\ast(Q)) $$
for any quadruple $Q$ with $cr(Q)=1$. Thus, $h^\ast=
\widehat{h}^\ast$. By a convergence property of Douady-Earle
extensions, we obtain
\begin{equation}\label{2-8}
||\mu_{ex(A_{1n}\circ h_n\circ
\gamma_n)}|_{D_0}-\mu_{ex(A_{2n}\circ h\circ
\gamma_n)}|_{D_0}||_{L^\infty }\rightarrow 0.
\end{equation}
On the other hand, (\ref{2-6}) and  the conformal naturality of
Douady-Earle extensions imply that for each $n$,
\begin{eqnarray*}
 & & ||\mu_{ex(A_{1n}\circ h_n\circ
\gamma_n)}|_{D_0}-\mu_{ex(A_{2n}\circ h\circ
\gamma_n)}|_{D_0}||_{L^\infty } \\
& = & ||\mu_{ex(h_n)}|_{D_n}-\mu_{ex(h)}|_{D_n}||_{L^\infty }\geq
\epsilon
\end{eqnarray*}
This is a contradiction to (\ref{2-8}). Thus,
$d_{AC}([h_n],[h])\rightarrow 0$ implies
$d_{AT}([h_n],[h])\rightarrow 0$ as $n\rightarrow \infty $.

\end{proof}

\section{A counter-example and a new lemma}

In this section, we construct a sequence $\{h_n\}_{n=2}^{\infty }$
of orientation-preserving homeomorphisms of $\uc $ such that for
each $n\ge 2$, $h_n$ fixes four points $\pm 1$ and $\pm i$ and the
Douady-Earle extension $ex(h_n)$ of $h_n$ fixes the origin, but
the maximal dilatation $K(ex(h_n))(0)$ of $ex(h_n)$ at the origin
goes to $\infty $ as $n\rightarrow \infty $. Let $A$ be the
M\"obius transformation from the unit disk $\ud $ to the upper
half plane $\h$ mapping $-1$, $-i$ and $1$ to $-1$, $0$ and $1$
respectively. Then the sequence $\{A\circ f_n\circ
A^{-1}\}_{n=2}^{\infty }$ is a counter-example to Lemma 2.2 in
\cite{Saric2} (the statement of this lemma is recalled in the
introduction).

 Given each $n\ge 2$, we
first define the map $h_n$ on the circular arc of $\uc $ in the
first quadrant; that is, $h_n$ maps the circular arc from $1$ to
$e^{\frac{\pi}{2}[1-\frac{1}{n}]i}$ linearly onto the arc from $1$
to $e^{\frac{\pi}{2n}i}$ and maps the arc from
$e^{\frac{\pi}{2}[1-\frac{1}{n}]i}$ to $i$ linearly onto the arc
from $e^{\frac{\pi}{2n}i}$ to $i$. Clearly, both $1$ and $i$ are
fixed under $h_n$ and $e^{\frac{\pi}{2}[1-\frac{1}{n}]i}$ is
mapped to $e^{\frac{\pi}{2n}i}$. Through conjugation under taking
conjugacies of complex numbers, $h_n$ is extended to be defined on
the arc of $\uc $ in the fourth quadrant. Finally, through
conjugation under taking mirror images with respect to the
imaginary axis, $h_n$ is extended to be defined on $\uc $. Note
that $h_2$ is equal to the identity map on $\uc $. Since each
$h_n, n\ge 2$, is symmetric with respect to the origin, using
definition we see that the Douady-Earle extension $ex(h_n)$ fixes
the origin. The main work is to show that
\begin{equation}\label{goesto1}
|Belt(ex(h_n))(0)|\rightarrow 1 \;\;{\rm as }\;\;n\rightarrow
\infty .
\end{equation}

Before starting to show this convergence, let us first recall some
formulas to express the Beltrami coefficient of the Douady-Earle
extension at the origin.

Let $h$ be an orientation-preserving homeomorphism of $\uc $. Assume
that the Douady-Earle extension $ex(h)$ is normalized at the origin;
that is, $ex(h)(0)=0$. Then the Beltrami coefficient of $ex(h)$ at
the origin can be explicitly expressed as follows
\cite{DouadyEarle}. Given any point $z\in \ud $, $ex(h)(z)$ is equal
to the unique point $w\in \ud $ such that
$$F(z, w)=0,$$ where
\begin{equation}\label{zw-equation}
F(z, w)=\frac{1}{2\pi }\int _{\uc }\frac{h(\xi )-w}{1-\bar{w}h(\xi
)}\cdot \frac{1-|z|^2}{|z-\xi |^2}|d\xi |.
\end{equation}
If we let
\begin{equation}\label{partialFz}
c_1=\frac{\partial F}{\partial z} (0, 0)=\frac{1}{2\pi }\int _{\uc
}\bar {\xi }h(\xi )|d\xi |,\;\;c_{-1}=\frac{\partial F}{\partial
\bar{z}} (0, 0)=\frac{1}{2\pi }\int _{\uc }\xi h(\xi )|d\xi |
\end{equation}
and
\begin{equation}\label{partialFw}
d_1=\frac{\partial F}{\partial w} (0,
0)=-1,\;\;d_{-1}=\frac{\partial F}{\partial \bar{w}} (0,
0)=\frac{1}{2\pi }\int _{\uc }h(\xi )^2|d\xi |,
\end{equation}
then
\begin{equation}\label{partialPhizbar}
\frac{\partial w}{\partial
\bar{z}}(0)=-\frac{\overline{\frac{\partial F}{\partial
w}(0,0)}\frac{\partial F}{\partial \bar{z}}(0, 0)-\frac{\partial
F}{\partial \bar{w}}(0,0)\overline{\frac{\partial F}{\partial
z}(0,0)}}{|\frac{\partial F}{\partial w}(0, 0)|^2-|\frac{\partial
F}{\partial \bar{w}}(0,
0)|^2}=\frac{c_{-1}+d_{-1}\overline{c}_1}{1-|d_{-1}|^2}.
\end{equation}
and
\begin{equation}\label{partialPhiz}
\frac{\partial w}{\partial z}(0)=-\frac{\overline{\frac{\partial
F}{\partial w}(0,0)}\frac{\partial F}{\partial z}(0,
0)-\frac{\partial F}{\partial
\bar{w}}(0,0)\overline{\frac{\partial F}{\partial \bar{z}}(0,
0)}}{|\frac{\partial F}{\partial w}(0, 0)|^2-|\frac{\partial
F}{\partial \bar{w}}(0,
0)|^2}=\frac{c_{1}+d_{-1}\overline{c}_{-1}}{1-|d_{-1}|^2}.
\end{equation}
Therefore the Beltrami coefficient $Belt(ex(h))$ at $0$ is equal to
\begin{equation}\label{Beltramicoeffi}
Belt(ex(h))(0)=\frac{c_{-1}+d_{-1}\overline{c}_1}{c_{1}+d_{-1}\overline{c}_{-1}}.
\end{equation}

If the boundary homeomorphism $h$ satisfies $h(\overline{\xi
})=\overline{h(\xi )}$ and $h(-\overline{\xi })=-\overline{h(\xi
)}$, then one can rewrite
\begin{equation}\label{partialFz-in-special-case}
c_1=\frac{2}{\pi }Re\int _0^{\frac{\pi}{2}}\bar {\xi }h(\xi )|d\xi
|,\;\;c_{-1}=\frac{2}{\pi }Re\int _0^{\frac{\pi}{2}}\xi h(\xi )|d\xi
|
\end{equation}
and
\begin{equation}\label{partialFw-in-special-case}
d_{-1}=\frac{2}{\pi }Re\int _0^{\frac{\pi}{2}}h^2(\xi )|d\xi |.
\end{equation}

Now let $h=h_n$. In the following, we show that both $c_1$ and
$c_{-1}$ converge to $\frac{2}{\pi}$ and $d_{-1}$ converges to $1$
as $n\rightarrow \infty $. It follows that $Belt(ex(h_n))(0)$
converges to $1$ as $n\rightarrow \infty $.

We first show that $c_1$ converges to $\frac{2}{\pi}$ as
$n\rightarrow \infty $.

Let $\theta _n=\frac{\pi }{2n}$, where $n\ge 2$. Then for each
$\xi =e^{i\theta }$ with $0\le \theta \le \frac{\pi }{2}-\theta
_n$, $h(\xi )=1+O(\frac{1}{n})$. Clearly,
$$\int _0^{\frac{\pi}{2}}\bar {\xi }h(\xi )|d\xi
|=\int _0^{\frac{\pi}{2}-\theta _n}\bar {\xi }h(\xi )|d\xi |+\int
_{\frac{\pi}{2}-\theta _n}^{\frac{\pi}{2}}\bar {\xi }h(\xi )|d\xi
|.$$ Then
\begin{eqnarray*}
\int _0^{\frac{\pi}{2}-\theta _n}\bar {\xi }h(\xi )|d\xi | & =
&\int _0^{\frac{\pi}{2}-\theta _n}\bar {\xi
}(1+O(\frac{1}{n}))|d\xi |=\int _0^{\frac{\pi}{2}-\theta _n}\bar
{\xi }|d\xi |+O(\frac{1}{n}) \\
& = &\int _0^{\frac{\pi}{2}}\bar {\xi }|d\xi
|+O(\frac{1}{n})=1-i+O(\frac{1}{n}) \end{eqnarray*} and
$$|\int
_{\frac{\pi}{2}-\theta _n}^{\frac{\pi}{2}}\bar {\xi }h(\xi )|d\xi
||\le \int _{\frac{\pi}{2}-\theta _n}^{\frac{\pi}{2}}|\bar {\xi
}h(\xi )||d\xi |=\theta _n=O(\frac{1}{n}).$$ Thus
$$\int _0^{\frac{\pi}{2}}\bar {\xi }h(\xi )|d\xi
|=1-i+O(\frac{1}{n})$$ and then
$$c_1=\frac{2}{\pi }+O(\frac{1}{n}).$$
It follows that $c_1$ converges to $\frac{2}{\pi}$ as
$n\rightarrow \infty $.

Similarly, one can show that $c_{-1}$ converges to $\frac{2}{\pi}$
as $n\rightarrow \infty $.

It remains to show that $d_{-1}$ converges to $1$ as $n\rightarrow
\infty $. Since for each $\xi =e^{i\theta }$ with $0\le \theta \le
\frac{\pi }{2}-\theta _n$, $h(\xi )=1+O(\frac{1}{n})$ implies
$h^2(\xi )=1+O(\frac{1}{n})$. Then
\begin{eqnarray*}
d_{-1} & = & \frac{2}{\pi }Re(\int _0^{\frac{\pi}{2}-\theta
_n}h^2(\xi )|d\xi |+\int _{\frac{\pi}{2}-\theta
_n}^{\frac{\pi}{2}}h^2(\xi )|d\xi |)\\
& = & \frac{2}{\pi }Re (\int _0^{\frac{\pi}{2}-\theta
_n}(1+O(\frac{1}{n}))|d\xi |+O(\frac{1}{n}))=\frac{2}{\pi }Re
(\frac{\pi }{2}+O(\frac{1}{n}))=1+O(\frac{1}{n}).
\end{eqnarray*}
Thus $d_{-1}$ converges to $1$ as $n\rightarrow \infty $. We
complete the proof of (\ref{goesto1}).

In the second half of this section, we introduce a new lemma which
is used in the next section to prove Theorem
\ref{improvedversion}. This bridges the gap in the proof of
Theorem A in \cite{Saric2}.

\begin{lemma}\label{newlemma}
Let $M>0$ and $h$ be an orientation-preserving homeomorphism of $\uc
$. Assume that $p_1=\frac{3+4i}{5}$, $p_2=-\bar{p}$, $p_3=-p$ and
$p_4=\bar{p}$. If $|\log cr(h(Q))|<M$ holds for any of the following
five quadruples
$$Q_0=\{-1, -i, 1, i\}, Q_1=\{-1, 1, p_1, i\}, Q_2=\{-1, 1, i, p_2\},$$
$$Q_4=\{-1, p_3, -i, 1\}\;{\rm and}\;Q_4=\{-1, -i, p_4, 1\},$$
then there exist a universal small neighborhood $U$ of the origin
and a constant $C>0$, only depending on $U$ and $M$, such that
$$K(ex(h)|_U)<C.$$
\end{lemma}

\begin{proof} Using the conformal naturality of Douady-Earle
extensions and invariance of the cross-ratio distortion norm of
$h$ under postcomposition by M\"obius transformations, we can
assume that the Douady-Earle extension $ex(h)$ fixes the origin.
Using Lemma 3.6 in \cite{Markovic}, the conclusion of this lemma
follows from the following claim.

\medskip
\noindent {\bf Claim.} {\em There exist $0<s_0<\pi $ (universal,
one may take $s_0=\frac{2\pi }{3}$) and $\epsilon >0$ (only
depending on $M$) such that for any circular arc $\Gamma $ of
length $>s_0$ and $<\pi $, the length of $h(\gamma )$ is greater
than $\epsilon $.}

\medskip
We divide the proof of this claim into the following four steps.

\medskip
\noindent {\bf Step 1.} In this step, we show there exists a
constant $\epsilon _1>0$ only depending on $M$ such that for each
half circle $\Gamma $ bounded by two of the four points $\{\pm 1,
\pm i\}$, the length $|h(\Gamma )|$ of $h(\Gamma )$ is greater
than $\epsilon _1$. For example, we show in detail how the two
conditions $ex(h)(0)=0$ and $|\log cr(h(Q_1))|<M$ imply that the
length $|h(\widehat{-i,1,i})|$ is greater than $\epsilon _1$.

Let $A=-1$, $B=-i$, $C=1$ and $D=i$, and let $A'=h(A)$, $B'=h(B)$,
$C'=h(C)$, $D'=h(D)$ and $p_k'=h(p_k)$, where $k=1, 2, 3, 4$.

Using the normalization condition $ex(h)(0)=0$, it is shown in Step
1 of the proof of Lemma 1 in \cite{HuMuzician1} that
$$|\widehat{A'B'C'D'}|\ge \frac{\pi }{3},$$
where $\widehat{A'B'C'D'}$ represents the image under $h$ of the
three quarter of the circle from $-1$ to $i$ and through $-i$ and
$1$ in the counterclockwise direction. Furthermore, it is shown in
Case 1 in Step 2 of the proof of Lemma 1 in \cite{HuMuzician1}
that
$$|\widehat{B'C'D'p_2'}|> \frac{\pi }{3},$$
where $\widehat{B'C'D'p_2'}$ represents the image under $h$ of the
circular arc from $-i$ to $p_2$ and through $1$ and $i$ in the
counterclockwise direction.

Now we apply the condition $|\log cr(h(Q_1))|<M$ to show that there
exists $\epsilon _1>0$ such that
$$|\widehat{B'C'D'}|>\epsilon _1.$$
Suppose that $l=|\widehat{B'C'D'}|$ is very small. Then
$l_1=|\widehat{C'p_1'D'}|$ is also very small since
$\widehat{C'p_1'D'}\subset \widehat{B'C'D'}$. It also follows that
$$|\widehat{D'p_2'}|=|\widehat{B'C'D'p_1'}|-|\widehat{B'C'D'}|>\frac{\pi}{3}-l$$ and
$$|\widehat{A'B'C'}|=|\widehat{A'B'C'D'}|-|\widehat{C'p_1'D'}|>\frac{2\pi}{3}-l_1>\frac{2\pi
}{3}-l.$$ Then
$$|\log cr(h(Q_1))|=|\log \frac{2\sin
\frac{|\widehat{C'p_1'D'}|}{2}\times 2\sin
\frac{|\widehat{p_2'A'}|}{2}}{2\sin
\frac{|\widehat{A'B'C'}|}{2}\times 2\sin
\frac{|\widehat{D'p_2'}|}{2}}|.$$ Clearly,
$$\frac{2\sin
\frac{|\widehat{C'p_1'D'}|}{2}\times 2\sin
\frac{|\widehat{p_2'A'}|}{2}}{2\sin
\frac{|\widehat{A'B'C'}|}{2}\times 2\sin
\frac{|\widehat{D'p_2'}|}{2}}<\frac{2\sin \frac{l}{2}\times 2}{2\sin
(\frac{\pi}{3}-\frac{l}{2})\times 2\sin
(\frac{\pi}{6}-\frac{l}{2})}.$$ Thus this quotient approaches $0$ as
$l\rightarrow 0$, which implies that $|\log cr(h(Q_1))|$ goes to
$\infty $ as $l\rightarrow 0$. This is a contradiction to $|\log
cr(h(Q_1))|<M$. Therefore, there exists $\epsilon _1>0$ such that
$|\widehat{B'C'D'}|>\epsilon _1$.

Similarly, using $ex(h)(0)=0$ and $|\log cr(h(Q_3))|<M$, we show
$|\widehat{C'D'A'}|>\epsilon _1$; using $ex(h)(0)=0$ and $|\log
cr(h(Q_4))|<M$, we show $|\widehat{D'A'B'}|>\epsilon _1$; using
$ex(h)(0)=0$ and $|\log cr(h(Q_1))|<M$, we show
$|\widehat{A'B'C'}|>\epsilon _1$.

\medskip
\noindent{\bf Step 2.} Using the estimates on the lengths of the
images of the four half circles obtained in Step 1 and the
condition that $|\log cr(h(Q_0))|<M$, the strategy to show the
existence of $\epsilon _1$ also imply that there exists $\epsilon
_2>0$ such that the length of the image under $h$ of each of the
four quarters $\widehat{Cp_1D}$, $\widehat{Dp_2A}$,
$\widehat{Ap_3B}$ and $\widehat{Bp_4C}$ of the circle is greater
than $\epsilon _2$.

\medskip
\noindent{\bf Step 3.} Using the estimates on the lengths of the
images of the four half circles obtained in Step 1, the estimates
on the lengths of the images of the four quarters obtained in Step
2, and the condition that $|\log cr(h(Q_k))|<M$ for $k=1, 2, 3,
4$, the strategy to show the existence of $\epsilon _1$ also imply
that there exists $\epsilon _3>0$ such that the length of the
image under $h$ of each of the eight circular arc bounded by two
adjacent points among $\{-1, p_3, -i, p_4, 1, p_1, i, p_2\}$ is
greater than $\epsilon _3.$

\medskip
\noindent{\bf Step 4.} Let $s_0=\frac{2\pi}{3}$ and $\epsilon
=\epsilon _3$. Then for any circular arc $\Gamma $ with length
$s_0<|\Gamma |<\pi$, $\Gamma $ contains at least one circular arc
between two adjacent points among $\{-1, p_3, -i, p_4, 1, p_1, i,
p_2\}$ since two longest adjacent arcs among those eight arcs
comprise an arc of length $2\arcsin \frac{4}{5}<\frac{2\pi }{3}$.
Therefore, the length $|h(\Gamma)|>\epsilon _3=\epsilon $.

\end{proof}

\section{The Farey tesselation, shear maps, and proof of Theorem \ref{improvedversion}}

We first introduce some background on the Farey tesselation and
shear maps. Then we prove Theorem \ref{improvedversion}.

In the following, we use the upper half plane $\h$ as a model for
the hyperbolic plane and recall the Farey tesselation and the
shear map, which were introduced by Penner \cite{Penner} and
furthered studied by Saric \cite{Saric2}\ \cite{Saric3}.

Let $\triangle_0$ be the idea geodesic triangle on $\h$ with
vertices $0,1$ and $\infty$ and let $\Gamma$ be the group
generated by the hyperbolic reflections to the sides of
$\triangle_0$. The Farey tesselation $\mathcal{F}$ is the
collection of the geodesics on the $\Gamma$-orbits of the edges of
$\triangle_0$. The set of the endpoints of the geodesics in $\f$
is equal to $\widehat{\mathbb{Q}}=\mathbb{Q}\cup\{\infty \}.$

Let $e\in \f$ and let $a,c\in \re$ be the endpoints of $e$. Assume
that $(\triangle_1,\triangle_2)$ is a pair of two triangles on
$\h$ with disjoint interiors and sharing a common boundary $e$.
Besides the vertices $a,c$,  we denote the third vertices of
$\triangle_1$ and $\triangle_2$ by $b$ and $d$ respectively. Given
a homeomorphism $h:\re\rightarrow\re$, the shear $s_h(e)$
mentioned in the introduction can be precisely defined by
$$s_h(e)=\log cr(\tilde{h}(a,b,c,d)),$$
where $cr(\tilde{h}(a,b,c,d))$ is the cross ratio defined by
(\ref{defofcr}) in the introduction. Then $h:\re \rightarrow\re $
induces a map $s_h: \f\rightarrow \mathbb{R}$, which is called the
shear map or function or coordinate of $\tilde{h}$. Conversely,
given any map $s:\f\rightarrow \mathbb{R}$, there is a unique
injective map $h_s$ from the vertices $\widehat{\mathbb{Q}}$ of
the Farey tesselation $\f$ into $\re$ such that $h_s$ fixes $0,1$
and $\infty$ and the shear map of $h_s$ is equal to $s$. We call
$h_s$ the characteristic map of $s$.

A fan $\f_p$ of geodesics in $\f$ with tip $p\in
\widehat{\mathbb{Q}}$ consists of all edges of $\f$ that have a
common endpoint at $p$. Each fan $\f_p$ has a natural order. Take
a horocycle $C$ tangent to $\re $ at $p$ and choose an orientation
on $C$ such that the corresponding horoball is to the left of $C$.
Let $e,e^\prime$ be two geodesics in $\f_p$. Define $e<e^\prime$
if the point $e\cap C$ comes before the point $e^\prime\cap C$.
This natural order on $\f_p$ gives a one-to-one correspondence
between $\f_p$ and the set $\mathbb{Z}$ of integers, and any two
such correspondences differ by a translation in $\mathbb{Z}$.

Given a shear map $s$ and a fan
$\f_p=\{e_{n}^p\}_{n\in\mathbb{Z}}$, the quantity $s(p;m,k)$ is
defined by (\ref{defofs-pmk}) in the introduction.

\begin{remark} {\em Let $s: \f \rightarrow \mathbb{R}$ be a map from
$\f $ to $\mathbb{R}$ and let $h_s$ be the characteristic map of
$s$. Assume that $e_j^p, j=m-k, m, m+k$, are three geodesics in a
fan $\f_p$. Let $a_{j}^p$ be the other endpoint of $e_j^p$ besides
$p$, where $j=m-k, m, m+k$. Using the definition of cross ratio
given by (\ref{defofcr}) and through pre-composition by a M\"obius
transformation to arrange $p$ at $\infty $, one can easily see
that for any $m, k\in \mathbb{Z}$,
$$cr(\{p,
a_{m-k}^p,a_{m}^p,a_{m+k}^p\})=1.$$ Similarly, through
post-composition by a M\"obius transformation to arrange $h_s(p)$
at $\infty $, we see
 $$s(p;m,k)=cr( h_s(\{p,
a_{m-k}^p,a_{m}^p,a_{m+k}^p\})).$$

Using the definitions of $d_S$ and $d_C$ in the introduction, we
know that for any $h_1, h_2\in T(\ud)$,
$$d_S(s_1,s_2)\leq d_C(h_1,h_2),$$
where $s_1$ and $s_2$ are the shear functions induced by $h_1$ and
$h_2$.

Substituting $h_2$ by $id$ in the previous inequality, one can see
that the shear norm $M_{s_1}$ of $h_1$ is finite is a necessary
condition for $s_1$ to be induced by a quasisymmetric map. This
verifies the necessity of the condition in Theorem A of
\cite{Saric2} for a shear function to be induced by a
quasisymmetric homeomorphism.

Using the previous inequality and Theorem
\ref{equi-of-three-metrics-on-T}, we can conclude that as
$n\rightarrow \infty $, $$d_T(h_n,h)\rightarrow 0 \text{ implies }
d_S(s_n,s)\rightarrow 0.$$  It is shown in Theorem B of
\cite{Saric2} that as $n\rightarrow \infty $,
$$d_T(h_n,h)\rightarrow 0 \text{ if and only if } d_S(s_n,s)\rightarrow 0.$$
}
\end{remark}

Let $A=-\frac{z-i}{z+i}$ be the M$\mathrm{\ddot{o}}$bius
transformation form $\h$ to $\ud$ that maps $0,1$ and $\infty$ to
$1,i$ and $-1$ respectively. In the following, given a
homeomorphism $h:\uc\rightarrow\uc$ of $\uc $, we denote by
$\widetilde{h}=A^{-1}\circ h\circ A$, which is a homeomorphism of
$\re$ fixing $0,1$ and $\infty$. Then Lemma \ref{newlemma}
converts to the following corollary.

\begin{corollary}\label{Controlunder8points} Let $M>0$ and $\widetilde{h}$ be a homeomorphism of $\re$ with $0,1$ and
$\infty$ fixed. If $|\log cr (\widetilde{h}(Q))|<M$ holds for any of
the following five quadruples
$$Q_0=\{\infty , -1, 0, 1\}, Q_1=\{\infty , 0, \frac{1}{2}, 1\}, Q_2=\{\infty , 0, 1,
2\},$$  $$Q_3=\{\infty , -2, -1, 0\} \;{\rm and}\;Q_4=\{\infty ,
-1, -\frac{1}{2}, 0\},$$ then there exist a constant $C>0$ and a
neighborhood $U$ of $i$, only depending on $M$, such that
$$K(ex(\widetilde{h})|_{U})<C.$$
\end{corollary}

\begin{proof} [Proof of Theorem \ref{improvedversion}] Given a homeomorphism $\tilde{h}_n$ of
$\re$, we denote by $s_n$ the shear map of $\tilde{h}_n$. Suppose
that Theorem \ref{improvedversion} fails. Then there exists a
sequence $\{\widetilde{h}_n\}_{n=1}^\infty$ of homeomorphisms of
$\re$ such that  $M_{s_n}<M$ for all $n$ and
$K(ex(\widetilde{h}_n))\rightarrow \infty $ as $n\rightarrow
\infty.$ Then there exists a sequence $\{z_n\}_{n=1}^{\infty }$ of
points on $\h$ such that
$|\mu_{ex(\widetilde{h}_n)}(z_n)|\rightarrow 1 $ as $n\rightarrow
\infty.$ Denote by $\triangle_n$ the ideal triangle in $\f$ such
that $z_n\in \overline{\triangle}_n$. Let $A_n\in
PSL(2,\mathbb{Z})$ such that $A_n(\triangle_n)=\triangle_0$, and
let $B_n\in PSL(2,\mathbb{R})$ such that $\widehat{h}_n=B_n\circ
\widetilde{h}_n \circ (A_n)^{-1}$ fixes $0,1$ and $\infty.$

Let $z_n^\prime=A_n(z_n)$. Consider the following two cases of
$\{z_n^\prime\}_{n=1}^\infty$. One is that a subsequence of
$\{z_n^\prime\}_{n=1}^\infty$ stays in a compact subset of
$\overline{\triangle}_0$, the other is that
$\{z_{n}\}_{n=1}^{\infty }$ has a subsequence converging to an
endpoint of $\triangle _0$ as $n\rightarrow\infty.$ In the second
case, without loss of generality we may assume that $z_{n}$
converges to $\infty$ as $n\rightarrow\infty.$

In the first case, we briefly denote the subsequence by
$\{z_n^\prime\}_{n=1}^\infty$ again. Since $M_{s_n}<M$ for all
$n,$ the sequence of the shear maps $s_n\circ A_n^{-1}$
corresponding to the homeomorphisms $\widehat{h}_n=B_n\circ
\widetilde{h}_n \circ (A_n)^{-1}$ has a convergent subsequence in
the sense that for any $e\in \f$, the sequence of real number
$s_{n_k}\circ A_{n_k}^{-1}(e)$ converges as $k\rightarrow\infty.$
It is clear that the limiting shear map $s_\infty:\f\rightarrow
\mathbb{R}$ satisfies $M_{s_\infty }<M.$ Since all maps
$\widehat{h}_{n_k}$ are normalized to fix three points, it is
shown on pages 2418-9 in \cite{Saric2} that the convergent
subsequence $\{\widehat{h}_{n_k}\}_{k=1}^{\infty }$ converges
pointwisely to a homeomorphism $\widehat{h}_\infty$ inducing the
shear map $s_\infty.$ Then $|\mu_{ex(\widehat{h}_{n_k})}|$
converges to $|\mu_{ex(\widehat{h}_\infty)}|$ uniformly on any
compact subset of $\h$. Since $ex(\widehat{h}_\infty)$ is a real
analytic diffeomorphism, it follows that for any compact $\Omega$
of $\h$, there exists $0<a<1$ such that
$|\mu_{ex(\widehat{h}_\infty)}|<a $ on $\Omega$. On the other
hand, by the conformal naturality of Douady-Earle extensions,
$$|\mu_{ex(\widehat{h}_{n_k})}(z_{n_k}^\prime)|=|\mu_{ex(\widetilde{h}_{n_k})}(z_{n_k})|\rightarrow
1 \ as\  k\rightarrow \infty.$$ This is a contradiction.

Now we show that there will be a contradiction coming up in the
second case too. Let $\lambda_{n}$ be the greatest even integer
less than or equal to $Im(z_n^\prime)$ and let
$\lambda_{n}^\prime$ be the real number such that
$\widehat{h}^\ast_{n}(x)=\lambda_{n}^\prime\widehat{h}_n(\lambda_{n}x)$
fixes $0,1$ and $\infty.$  We show that $\widehat{h}^\ast_{n}$
satisfies the conditions in Corollary 1 for all $n$. For example,
we verify that if  $Q=\{\infty,-1,0,1\}$, then for all $n$,
$$|\log cr(\widehat{h}^\ast_{n}(Q))|\leq M.$$
Verifications of the above inequality for the other four cases of
$Q$ in Corollary 1 are similar. Therefore, we skip them here.

Let $e_k^\infty$ be the edge in $\f$ that connects the integer $k$
to $\infty$. Using the fan $\{e_k^\infty\}_{k\in \mathbb{Z}}$, we
obtain
$$\begin{array}{lll} |\log cr(\widehat{h}^\ast_{n}(\{\infty,-1,0,1\}))|
&=&|\log
cr(\widehat{h}_{n}(\{\infty,-\lambda_{n},0,\lambda_{n}\}))|\\
&=& |\log cr(\tilde{h}_n\circ A_{n}^{-1}(\{\infty, -\lambda _n,
0,\lambda_{n}\}))|\leq M_{s_n}\leq M.
\end{array}$$
Let $w_n=(1/\lambda_{n})z_n^\prime$. Then $w_n$ converges to the
point $i$ as $n\rightarrow \infty $. Therefore,
$$|\mu_{ex(\widehat{h}^\ast_{n})}(w_n)|=|\mu_{ex(\widehat{h}_n)}(z_n^\prime)|=
|\mu_{ex(\widetilde{h}_n)}(z_n)|\rightarrow 1$$  as $n\rightarrow
\infty.$ This is a contradiction to Corollary
\ref{Controlunder8points}.

We complete the proof.
\end{proof}

\section{A metric characterization of $AT(\ud )$ through
equivalent classes of shear maps}

For each edge $e$ in $\f$, we define the Farey generation $g(e)$
of $e$ as follows. Each boundary edge of $\triangle_0$ has Farey
generation 0. If an edge $e\in \f$ is obtained by $n$ reflections
of an edge of $\triangle_0$ (where $n$ is the smallest such
number), then the Farey generation $g(e)$ of $e$ is defined to be
$n$.

Given any two points $h_1,h_2\in T(\ud)$,  let $s_1$ and $s_2$ be
the shear maps induced by  $h_1$ and $h_2$ respectively. For each
fan $\f_p=\{e_{n}^p\}_{n\in\mathbb{Z}}$, we let
$$AM_{s_1,s_2}(p)=\limsup_{\min\{g(e^p_{m+k}),g(e^p_{m-k})\}\rightarrow \infty}
|\log\frac{s_1(p;m,k)}{s_2(p;m,k)}|,$$ and
$$d_{AS}{(s_1,s_2)}=\sup_{p\in \mathcal{P}} AM_{s_1,s_2}(p),$$
where $\mathcal{P}$ is the collection of all endpoints of the
geodesics in $\f$.

Recall that $\x$ is the collection of all real-valued functions
$s:\f \rightarrow \mathbb{R}$ with finite shear norms, and $\x$
gives a parametrization of $T(\ud )$. Then it is clear that
$d_{AS}$ defines a pseudo metric on $\mathcal{X}$ and hence
introduces an equivalent relation on $\mathcal{X}$ by defining two
elements $s$ and $s^\prime $ of $\mathcal{X}$ to be equivalent if
$d_{AS}(s, s^\prime )=0$. We denote by $[s]$ the equivalent class
of $s$ and by $\mathcal{AX}$ the quotient space under this
equivalent relation. The pseudo metric $d_{AS}$ on $\x $ induces a
a metric on $\mathcal{AX}$, which is called the {\em asymptotic
shear metric} on $\mathcal{AX}$ and still denoted by $d_{AS}$. In
this section, we prove our Theorems
\ref{charactrization-via-shears} and \ref{top-equiv}, that is, we
show that $\mathcal{AX}$ is a parametrization of $AT(\ud )$ and
$(\mathcal{AX}, d_{AS})$ is topologically equivalent to $(AT(\ud
), d_{AT})$, where $d_{AT}$ is the \te \ metric on $AT(\ud )$.

\begin{remark} {\em Let $e_j^p$ be an edge in $\f$ with tip $p$ and let $a_{j}^p$ be the
endpoint of $e_j^p$ not equal to $p$. Note first that
$\min\{g(e^p_{m+k}),g(e^p_{m-k})\}\rightarrow \infty$ is
equivalent to say that the minimal scale of the corresponding
quadruple $\{p, a_{m-k}^p,a_{m}^p,a_{m+k}^p\}\rightarrow 0$. From
 the definitions of $d_{AS}$ and $d_{AC}$, it follows that for any
two points $[h_1]$ and $[h_2]$ in $AT(\ud )$,
$$d_{AS}([s_1],[s_2])\leq d_{AC}([h_1],[h_2]).$$
From Theorem \ref{equi-of-three-metrics-on-AT}, we know that
$d_{AT}([h_n],[h])\rightarrow 0$ implies
$d_{AS}([s_n],[s])\rightarrow 0$ as $n\rightarrow \infty $. Our
Theorem \ref{top-equiv} states that $d_{AT}([h_n],[h])\rightarrow
0$ if and only if $d_{AS}([s_n],[s])\rightarrow 0$ as
$n\rightarrow \infty $. }
\end{remark}

\begin{remark}
{\em If $h_2$ is the identity map, then $s_2\equiv 0$ and
$s_2(p;m,k)\equiv 1$. Thus
$$ d_{AS}{(s_1, 0)}=\sup_{p\in \mathcal{P}}\limsup_{\min\{g(e^p_{m+k}),g(e^p_{m-k})\}\rightarrow \infty}
|\log s_1(p;m,k)|.$$ It has already been proved in Theorem A of
\cite{Saric2} that $h_1$ is symmetric if and only if $d_{AS}{(s_1,
0)}=0$. This means the equivalent class containing the identity
map is the collection $S(\uc )$ of all symmetric homeomorphisms in
$QS(\uc )$. This is a special case of our Theorem
\ref{charactrization-via-shears}.}
\end{remark}

\begin{proof} [Proof of Theorem \ref{charactrization-via-shears}] Given two points $h$ and $h^\prime$ in $T(\ud ),$  let $s$
and $s^\prime$ be the shear maps of $\widetilde{h}=A^{-1}\circ
h\circ A$ and $\widetilde{h}^\prime=A^{-1}\circ h^\prime\circ A$
respectively. If $[h]=[h^\prime],$ then it follows from Remark 2
and Theorem \ref{equi-of-three-metrics-on-AT} that
$d_{AS}(s,s^\prime)=0.$

In the following, we apply some ideas in \cite{Saric2} to show
that $d_{AS}(s,s^\prime)=0$ implies $[h]=[h^\prime].$ Suppose on
the contrary this is not true. Then there exist two points $h$ and
$h'$ in $T(\ud )$ such that $d_{AS}(s,s^\prime)=0$ but
$[h]\neq[h^\prime]$. Using the same notation in the previous
section, we let $A=-\frac{z-i}{z+i}$ and
$\widetilde{h}=A^{-1}\circ h\circ A$ and
$\widetilde{h}'=A^{-1}\circ h'\circ A$. Denote by
$ex(\widetilde{h})$ and $ex(\widetilde{h^\prime})$ the
Douady-Earle extensions of $\widetilde{h}$ and
$\widetilde{h^\prime}$ respectively. Since $[h]\neq[h^\prime]$,
there exist a constant $c>0$ and a sequence of points $z_n\in \h$
leaving every compact subset of $\h$ such that
\begin{equation}\label{5-1}
|\mu_{ex(\widetilde{h})}(z_n)-\mu_{ex(\widetilde{h^\prime})}(z_n)|\geq
c.
\end{equation}
Denote by $\triangle_n$ the ideal triangle in $\f$ such that
$z_n\in \overline{\triangle}_n$ and by $A_n$ the element in
$PSL(2,\mathbb{Z})$ such that $A_n(\triangle_n)=\triangle_0$. Let
$B_n$ and $B_n^\prime$ be the elements in $PSL(2,\mathbb{R})$ such
that $\widetilde{h}_n=B_n\circ \widetilde{h} \circ (A_n)^{-1}$ and
$\widetilde{h}^\prime_n=B_n^\prime\circ \widetilde{h}^\prime \circ
(A_n)^{-1}$ fixing $0,1$ and $\infty .$

Let $z_n^\prime=A_n(z_n)$. It suffices to consider the following
two cases: one is that $\{z_n^\prime\}_{n=1}^\infty$ has a
subsequence staying in a compact set of $\overline{\triangle}_0$
and the other is that $z_{n}$ has a subsequence converging to one
of the endpoint of $\triangle _0$ as $n\rightarrow\infty.$ In the
second case, without loss of generality we assume that $z_n$
converges to $\infty $ as $n\rightarrow \infty $.

In the first case, for brevity in notation we continue to denote
by $\{z_n^\prime\}_{n=1}^\infty$ the subsequence staying in a
compact set of $\overline{\triangle}_0$. Given an edge $e\in \f$,
we claim that the Farey generation $g(A_n^{-1}(e))$ goes to
$\infty$ as $n\rightarrow \infty.$ This is clear if $e$ is an edge
$e_0$ of $\triangle_0$ (since $A_n^{-1}(e_0)$ is an edge of
$\triangle_n$). For any other edge $e\in \f $, we note that for
any $n\in \mathbb{N}$,
$$g(A_n^{-1}(e))-g(A_n^{-1}(e_0))=g(e)-g(e_0),$$
which implies that $g(A_n^{-1}(e))\rightarrow \infty $ as
$n\rightarrow \infty $. Now using the assumption
$d_{AS}(s,s^\prime)=0$, we know for each $e\in \f$, $s\circ
A_n^{-1}(e)$ and $s^\prime\circ A_n^{-1}(e)$ converge to the same
limiting value. Therefore, $\{\widetilde{h}_n\}_{n=1}^\infty$ and
$\{\widetilde{h}_n^\prime\}_{n=1}^\infty$ converge to the same
limiting homeomorphism pointwisely. Since
$\{K(ex(\widetilde{h}_n))\}_{n=1}^\infty$ and
$\{K(ex(\widetilde{h}_n^\prime))\}_{n=1}^\infty$ are bounded and
all $\tilde{h}_n$ and $\tilde{h}_n^\prime$ are normalized to fix
three points, passing to subsequences we may assume that
$\{\widetilde{h}_n\}_{n=1}^\infty$ and
$\{\widetilde{h}^\prime_n\}_{n=1}^\infty$ converge uniformly to
the same quasisymmetric homeomorphism of $\re$. Thus,
$$|\mu_{ex(\widetilde{h})}(z_n)-\mu_{ex(\widetilde{h^\prime})}(z_n)|=
|\mu_{ex(\widetilde{h}^\prime_n)}(z_n^\prime)-\mu_{ex(\widetilde{h}_n)}(z_n^\prime)|\rightarrow
0$$ as $n\rightarrow\infty$. This is a contradiction to
(\ref{5-1}).

In the second case, let $\lambda_{n,1}$ be the greatest integer
less than or equal to $Im(z_n^\prime)$. Then
$\lambda_{n,1}\rightarrow\infty$ as $n\rightarrow\infty$. Let
$\beta_{n,1}$ and  $\beta_{n,1}^\prime$ be the two real numbers
such that $\widehat{h}_{n,1}(x)=\beta_{n,1} \widetilde{h}_n
(\lambda_{n,1} x)$ and
$\widehat{h}_{n,1}^\prime(x)=\beta_{n,1}^\prime \widetilde{h}'_n
(\lambda_{n,1} x)$ fix $0,1$ and $\infty.$

Let $e_k^\infty$ be the edge in $\f$ connecting an integer $k$ to
$\infty$. At first we claim that there exists at most one $k_0\in
\mathbb{Z}$ such that the Farey generation
$g(A_n^{-1}(e_{\lambda_{n,1}k_0}^\infty))$ is bounded as
$n\rightarrow\infty.$ In the fan $\{e^\infty_k\}_{k\in
\mathbb{Z}}$, there are $|\lambda_{n,1}k_0-\lambda_{n,1}k|$ edges
between $e^{\infty }_{\lambda_{n,1}k_0}$ and $e^{\infty
}_{\lambda_{n,1}k}.$ Given any two integers $k\neq k_0,$ since
$|\lambda_{n,1}k_0-\lambda_{n,1}k|\rightarrow\infty$ as
$n\rightarrow\infty $, it follows that
$g(A_n^{-1}(e_{\lambda_{n,1}k}^\infty))$ is unbounded if
$g(A_n^{-1}(e_{\lambda_{n,1}k_0}^\infty))$ is bounded. Thus, the
claim follows.

We now show that for all  $k\in \mathbb{Z}$,
\begin{equation}\label{5-2}
\lim_{n\rightarrow \infty}\widehat{h}_{n,1}(k)= \lim_{n\rightarrow
\infty}\widehat{h}_{n,1}^\prime(k).\end{equation} We divide the
proof into three cases according to the value of $k_0$: $k_0=0$,
$k_0<0$ or $k_0>0$.

If $k_0=0,$ then the Farey generations
$g(A_n^{-1}(e_{-\lambda_{n,1}}^\infty))$ and
$g(A_n^{-1}(e_{\lambda_{n,1}}^\infty))$ go to $\infty$ as
$n\rightarrow\infty.$ Since $d_{AS}(s,s^\prime)=0$, it follows
from Remarks 1 and 2 that
$$\begin{array}{lr}
\frac{s\circ A_n^{-1}(\infty, 0, \lambda_{n,1})}{s^\prime \circ
A_n^{-1}(\infty, 0, \lambda_{n,1})}=
\frac{cr(\widetilde{h}^\prime_n(-\lambda_{n,1},0,\lambda_{n,1},\infty))}
{cr(h_n(-\lambda_{n,1},0,\lambda_{n,1},\infty))}\\
 =\frac{cr(\widehat{h}_{n,1}^\prime(-1,0,1,\infty))}
{cr(\widehat{h}_{n,1}(-1,0,1,\infty))}=\frac{\widehat{h}_{n,1}^\prime(-1)}{\widehat{h}_{n,1}(-1)}
\rightarrow 1 \ \mathrm{as }\ n\rightarrow\infty.
\end{array}$$
Thus, $$\lim_{n\rightarrow \infty}\widehat{h}_{n,1}(-1)=
\lim_{n\rightarrow \infty}\widehat{h}_{n,1}^\prime(-1).$$ Using
the symmetric triple $(-1,1,3)$, the assumption
$\widehat{h}_{n,1}(1)=\widehat{h}_{n,1}^\prime(1)$ and the
equality of the above two limits, we first obtain
$$\lim_{n\rightarrow \infty}\widehat{h}_{n,1}(3)=
\lim_{n\rightarrow \infty}\widehat{h}_{n,1}^\prime(3).$$ Then we
use $(1,2,3)$ and then $(2,3,4)$ and so on to obtain that
$$\lim_{n\rightarrow \infty}\widehat{h}_{n,1}(k)=
\lim_{n\rightarrow \infty}\widehat{h}_{n,1}^\prime(k)$$ for each
$k\in \mathbb{N}.$ Thirdly, for each $k\in \mathbb{N}$, we use the
triple $(-k,0,k)$ to obtain that $$\lim_{n\rightarrow
\infty}\widehat{h}_{n,1}(-k)= \lim_{n\rightarrow
\infty}\widehat{h}_{n,1}^\prime(-k).$$  Using the normalization
condition, it is clear that
$$\lim_{n\rightarrow
\infty}\widehat{h}_{n,1}(0)= \lim_{n\rightarrow
\infty}\widehat{h}_{n,1}^\prime(0).$$ Therefore, (\ref{5-2}) holds
if $k_0=0$.

If $k_0<0,$ we first use the triple $(0,1,2)$ to obtain
$$\lim_{n\rightarrow \infty}\widehat{h}_{n,1}(2)=
\lim_{n\rightarrow \infty}\widehat{h}_{n,1}^\prime(2).$$ Using the
same argument as in the case when $k_0=0$, we obtain for each
$k\in \mathbb{N}$,
$$\lim_{n\rightarrow \infty}\widehat{h}_{n,1}(k)=
\lim_{n\rightarrow \infty}\widehat{h}_{n,1}^\prime(k).$$ Thirdly,
for each $k\in \mathbb{Z}\backslash k_0$, we use $(-k,0,k)$ to
obtain that $$\lim_{n\rightarrow \infty}\widehat{h}_{n,1}(k)=
\lim_{n\rightarrow \infty}\widehat{h}_{n,1}^\prime(k).$$  Finally,
we use $(k_0-1,k_0,k_0+1)$ to obtain that $$\lim_{n\rightarrow
\infty}\widehat{h}_{n,1}(k_0)= \lim_{n\rightarrow
\infty}\widehat{h}_{n,1}^\prime(k_0).$$ Therefore, we show
(\ref{5-2}) if $k_0<0$.

The proof of (\ref{5-2}) for $k_0>0$ is very similar to what we
have done in the case $k_0<0$. We skip it.

For each $r\in \mathbb{N}$, let $\lambda_{n,r}$ be the greatest
integral multiple of $2^{r-1}$ less than or equal to
$Im(z_n^\prime)$. Let $\beta_{n,r}$ and  $\beta_{n,r}^\prime$ be
the real numbers such that $\widehat{h}_{n,r}(x)=\beta_{n,r}
\widetilde{h}_n (\lambda_{n,r} x)$ and
$\widehat{h}_{n,r}^\prime(x)=\beta_{n,r}^\prime \widetilde{h}_n
(\lambda_{n,r} x)$ fix $0,1$ and $\infty.$  Using the same
argument in the case when $r=1$, it is easy to see that there
exists at most one integer $k_{r-1}\in \mathbb{Z}$ such that the
Farey generation $g(A_n^{-1}(e_{\lambda_{n,r}k_{r-1}}^\infty))$ is
bounded as $n\rightarrow\infty$. Now we claim that for each $r\in
\mathbb{N}$ and each $k\in \mathbb{Z}$,
\begin{equation}\label{5-3}
\lim_{n\rightarrow \infty}\widehat{h}_{n,r}(\frac{k}{2^{r-1}})=
\lim_{n\rightarrow
\infty}\widehat{h}_{n,r}^\prime(\frac{k}{2^{r-1}}). \end{equation}
This claim can be proved by an induction on $r$. We have just
shown the case when $r=1$. Assume that the claim is true for
$r=m$, we only need to show the claim is true for $r=m+1$ and $k$
is odd. Given each $k\in \mathbb{Z},$ we can choose an odd integer
$k^\prime\in \mathbb{Z}$ such that both
$\frac{k-k^\prime}{2^{m-1}}$ and $\frac{k+k^\prime}{2^{m-1}}$ are
not equal to $k_{r-1}$. Since the limits of $\widehat{h}_{n,r}$
and $\widehat{h}_{n,r}^\prime$ are the same at the points
$\frac{k-k^\prime}{2^{m-1}}$ and $\frac{k-k^\prime}{2^{m-1}}$,
using the symmetric triple $(\frac{k-k^\prime}{2^{m-1}},
\frac{k}{2^{m}}, \frac{k+k^\prime}{2^{m-1}})$ we obtain the claim
(\ref{5-3}) for $r=m+1$ and an odd integer $k$.

For each $r\in \mathbb{N}$, since
$\{K(ex(\widehat{h}_{n,r}))\}_{n=1}^\infty$ and
$\{K(ex(\widehat{h}_{n,r}^\prime))\}_{n=1}^\infty$ are bounded and
all the maps in these two sequences are normalized to fix three
points, passing to subsequences we may assume that
$\{\widehat{h}_{n,r}\}_{n=1}^\infty$ and
$\{\widehat{h}_{n,r}^\prime\}_{n=1}^\infty$ converge uniformly to
two quasisymmetric homeomorphisms of  $\re$, which agree on
$\{\frac{k}{2^{r-1}}: k\in \mathbb{Z}.\}$ because of the previous
claim (\ref{5-3}). Then for each $r\in \mathbb{N}$, we choose a
sufficiently big $n_r$ such that for each $k\in \mathbb{Z}$,
$$|\widehat{h}_{n_r,r}(\frac{k}{2^{r-1}})-\widehat{h}_{n_r,r}^\prime(\frac{k}{2^{r-1}})|\leq
1/r.$$ Since $Im(z_{n_r}^\prime)\rightarrow \infty $ as
$n\rightarrow \infty $, we can also choose $n_r$ large enough such
that $1\leq Im(z_{n_r}^\prime)/\lambda_{n_r,r}<2$. Then
$\{\frac{z_{n_r}^\prime}{\lambda_{n_r,r}}\}_{r=1}^{\infty }$ stays
in a compact subset of $\mathbb{H}$. Because $\{\frac{k}{2^r}:
k\in \mathbb{Z}, r\in \mathbb{N}\}$ is dense in $\re$, the two
sequences $\{\widehat{h}_{n_r,r}\}_{r=1}^\infty$ and
$\{\widehat{h}_{n_r,r}^\prime\}_{r=1}^\infty$ converge pointwisely
to the same quasisymmetric homeomorphism of $\re$. Using a
property of Douady-Earle extensions, we conclude that
\begin{equation}\label{5-4}
|\mu_{ex(\widehat{h}_{n_r,r})}
(\frac{z_{n_r}^\prime}{\lambda_{n_r,r}})-
\mu_{ex(\widehat{h}^\prime_{n_r,r})}(\frac{z_{n_r}^\prime}{\lambda_{n_r,r}})|\rightarrow
0 \;\;\text{as}\;\;r\rightarrow \infty .
\end{equation}
On the other hand, using (\ref{5-1}) we obtain for each $r\in
\mathbb{N}$,
\begin{equation*}
\begin{array}{lll} |\mu_{ex(\widehat{h}_{n_r,r})}
(\frac{z_{n_r}^\prime}{\lambda_{n_r,r}})-
\mu_{ex(\widehat{h}^\prime_{n_r,r})}(\frac{z_{n_r}^\prime}{\lambda_{n_r,r}})|
&=&|\mu_{ex(\widetilde{h}_{n_r,r})}(z^\prime_{n_r,r})-
\mu_{ex(\widetilde{h}^\prime_{n_r,r})}(z^\prime_{n_r,r})|\\
&=&
|\mu_{ex(\widetilde{h})}(z_{n_r,r})-\mu_{ex(\widetilde{h^\prime})}(z_{n_r,r})|\geq
c>0.
\end{array}
\end{equation*}
This is a contradiction to (\ref{5-4}).
\end{proof}

\begin{proof} [Proof of Theorem \ref{top-equiv}] By Theorem 7 and Remark 2, we
see that $d_{AT}([h_n],[h])\rightarrow 0$ implies
$d_{AS}([s_n],[s])\rightarrow 0$ as $n\rightarrow \infty$.

It remains to show that $d_{AS}([s_n],[s])\rightarrow 0$ implies
$d_{AT}([h_n],[h])\rightarrow 0$ as $n\rightarrow \infty$. Suppose
that on the contrary this implication is not true. Then there
exist $[h]\in AT(\ud )$ and a sequence $\{[h_n]\}_{n=1}^{\infty }$
in $AT(\ud )$ such that $d_{AS}([s_n], [s])$ goes to $0$ but
$d_{AT}([h_n],[h])$ dose not go to $0$ as $n\rightarrow \infty$.
Passing to a subsequence, we may assume that
$$d_{AT}([h_n],[h])\geq c'$$ for some $c'>0$ and all $n$. Then
there exist a constant $c>0$ and a sequence of points $z_n\in \h$
leaving every compact subset of $\h$ such that
\begin{equation}\label{5-6}
|\mu_{ex(\widetilde{h})}(z_n)-\mu_{ex(\widetilde{h}_n)}(z_n)|\geq
c.
\end{equation}
Using a very similar strategy in the proof of the previous
theorem, a contradiction can be derived as follows.

Denote by $\triangle_n$ the ideal triangle in $\f$ such that
$z_n\in \overline{\triangle}_n$ and by $A_n$ the element in
$PSL(2,\mathbb{Z})$ such that $A_n(\triangle_n)=\triangle_0$. Let
$B_n$ and $B_n^\prime$ be the elements in $PSL(2,\mathbb{R})$ such
that $\widehat{h}_n=B_n\circ \widetilde{h} \circ (A_n)^{-1}$ and
$\widehat{h}^\prime_n=B_n^\prime\circ \widetilde{h}_n\circ
(A_n)^{-1}$ fix $0,1$ and $\infty.$ Let $z_n^\prime=A_n(z_n)$.
Consider the following two cases: one is that the sequence
$\{z_n^\prime\}_{n=1}^\infty$ has a subsequence staying in a
compact set of $\overline{\triangle}_0$ and the other is that
without loss of generality, we may assume $z_{n}$ goes to $\infty$
as $n\rightarrow\infty.$

In the first case, we denote the subsequence by
$\{z_n^\prime\}_{n=1}^\infty$ again. For each  $e\in \f$,
$g(A_n^{-1}(e))$ goes to $\infty$ as $n\rightarrow \infty.$ Using
the assumption that $d_{AS}([s_n],[s])$ approaches $0$ as
$n\rightarrow \infty $, we see that $s\circ A_n^{-1}(e)$ and
$s_n\circ A_n^{-1}(e)$ converge to the same limit for each $e\in
\f$. Using
 the normalization condition, we conclude that $\{\widehat{h}_n\}_{n=1}^\infty$
 and
$\{\widehat{h}^\prime_n\}_{n=1}^\infty$ converge pointwisely to
the same quasisymmetric homeomorphism of $\re$. Using a property
of Douady-Earle extensions,
\begin{equation}\label{5-7}
|\mu_{ex(\widehat{h}^\prime_n)}(z^\prime_n)-\mu_{ex(\widehat{h}_n)}(z^\prime_n)|\rightarrow
0 \;\;\text{as}\;\;n\rightarrow \infty .
\end{equation}
On the other hand, (\ref{5-6}) implies
\begin{equation*} |\mu_{ex(\widehat{h}^\prime_n)}(z^\prime_n)-\mu_{ex(\widehat{h}_n)}(z^\prime_n)|
=|\mu_{ex(\widetilde{h})}(z_n)-\mu_{ex(\widetilde{h}_n)}(z_n)|
\geq c>0.
\end{equation*}
This is a contradiction to (\ref{5-7}).

In the second case, for each $r\in \mathbb{N}$, let
$\lambda_{n,r}$ be the greatest integral multiple of $2^{r-1}$
less than or equal to $Im(z_n^\prime)$. Let $\beta_{n,r}$ and
$\beta_{n,r}^\prime$ be the real numbers such that
$\widehat{h}_{n,r}(x)=\beta_{n,r} \widetilde{h}_n (\lambda_{n,r}
x)$ and $\widehat{h}_{n,r}^\prime(x)=\beta_{n,r}^\prime
\widetilde{h}_n (\lambda_{n,r} x)$ fix $0,1$ and $\infty.$ Using
the assumption that $d_{AS}([s_n],[s])\rightarrow0$ as
$n\rightarrow \infty $ and following the exact same details used
to derive a contradiction for the second case in the proof of
Theorem \ref{charactrization-via-shears}, we can conclude that for
each $r\in \mathbb{N}$, $ \{\widehat{h}_{n,r}\}_{n=1}^\infty$ and
$\{\widehat{h}_{n,r}^\prime\}_{n=1}^\infty$ converge pointwisely
to two quasisymmetric homeomorphisms of $\re $ that agree on the
set $\{\frac{k}{2^{r-1}}, k\in \mathbb{Z}\}$. Then the part of the
last paragraph in the proof of Theorem
\ref{charactrization-via-shears} after the first sentence also
forms the rest of the proof for Theorem \ref{top-equiv}.
\end{proof}

\bibliographystyle{amsplain}

\end{document}